\documentclass{amsart}

\usepackage[hyperref,all]{bi-discrete}

\usepackage[backend=biber,maxbibnames=5,maxalphanames=5,style=alphabetic,bibencoding=utf8,giveninits,url=false,isbn=false]{biblatex}
\AtBeginBibliography{\small}

\usepackage{tikz}
\usetikzlibrary{calc}
\usepackage{pgfplots}
\pgfplotsset{compat = newest}

\usepackage{mathtools}

\pgfplotscreateplotcyclelist{MyColorsPlotU}{%
    {green,mark = *          ,every mark/.append style={solid,scale=0.6,fill=green}},
    {blue,mark = diamond*			 ,every mark/.append style={solid,scale=0.8,fill=blue}},
    {red,mark = square*			 ,every mark/.append style={solid,scale=0.8,fill=red}},
    {violet,mark = halfcircle*  ,every mark/.append style={solid,scale=0.8,fill=violet}},
    {dashed,black},
  }
  
  \pgfplotscreateplotcyclelist{MyColorsPlotU3}{%
    {red,mark = square*			 ,every mark/.append style={solid,scale=0.8,fill=red}},
    {violet,mark = halfcircle*  ,every mark/.append style={solid,scale=0.8,fill=violet}},
    {dashed,black},
  }

\pgfplotscreateplotcyclelist{MyColorsPlotU2}{%
    {green,mark = *          ,every mark/.append style={solid,scale=0.4,fill=green}},
    {blue,mark = diamond*			 ,every mark/.append style={solid,scale=0.4,fill=blue}},
    {red,mark = square*			 ,every mark/.append style={solid,scale=0.4,fill=red}},
    {violet,mark = halfcircle*  ,every mark/.append style={solid,scale=0.4,fill=violet}},
    {dashed,black},
  }

\providecommand{\tria}{\mathcal{T}}
\providecommand{\dx}{\, \mathrm{d}x}

\begin{document}

\author[J.~Storn]{Johannes Storn}
\address[J.~Storn]{Department of Mathematics, University of Bielefeld, Postfach 10 01 31, 33501 Bielefeld, Germany}
\email{jstorn@math.uni-bielefeld.de}
\thanks{%
\textit{Funding.} The work of the author was supported by the Deutsche Forschungsgemeinschaft (DFG, German Research Foundation) -- SFB 1283/2 2021 -- 317210226.\\ \indent
\textit{Data Availability.}
The datasets generated during and/or analysed during the current study are available from the corresponding author on reasonable request.}


\title[Minimal Residual methods in negative norms]{Solving Minimal Residual Methods in $W^{-1,p'}$ with large Exponents $p$} 
\begin{abstract}
We introduce a numerical scheme that approximates solutions to linear PDE's by minimizing a residual in the $W^{-1,p'}(\Omega)$ norm with exponents $p> 2$. 
The resulting problem is solved by regularized \Kacanov{} iterations, allowing to compute the solution to the non-linear minimization problem even for large exponents $p\gg 2$. Such large exponents remedy instabilities of finite element methods for problems like convection-dominated diffusion.
\end{abstract}
\maketitle

\section{Introduction}
Starting with the work of Guermond \cite{Guermond04}, recent papers like \cite{HoustonMugaRoggendorfvanderZee19,MugaZee20,HoustonRoggendorfVanDerZee22,LiDemkowicz22,MillarMugaRojasVanDerZee22} have approximated linear PDE's by minimal residual methods in Banach spaces. The reasons for using Banach spaces like $L^p(\Omega)$ or $W^{1,p}_0(\Omega)$ rather than Hilbert spaces like $L^2(\Omega)$ or $W^{1,2}_0(\Omega)$ are manyfold. For example, rough data might lead to solutions that are not in $L^2(\Omega)$, see \cite{HoustonMugaRoggendorfvanderZee19,MillarMugaRojasVanDerZee22}.
Furthermore, minimizing in $L^1(\Omega)$ seems to allow for computations of viscosity solutions, see \cite[Sec.~4.6]{Guermond04}. Moreover, classical finite element methods lead for problems like singular perturbed problems or convection-dominated diffusion to non-physical oscillations that can be overcome by the use of minimal residual methods in spaces like $L^1(\Omega)$, see for example \cite{HoustonRoggendorfVanDerZee22}.
Unfortunately, the resulting numerical schemes are non-linear minimization problems which are difficult to solve.
We overcome this downside for a minimal residual method in $W^{-1,p'}(\Omega)$ with $p>2$ by introducing a regularized \Kacanov{} scheme that converges even for large exponents $p\gg 2$ towards the exact discrete minimizer.
More precisely, we do the following.

Suppose we have some linear operator $B\colon W^{1,p'}_0(\Omega) \to W^{-1,p'}(\Omega)$ that maps the Sobolev space $W^{1,p}_0(\Omega)$ equipped with homogeneous Dirichlet boundary data onto the dual space $W^{-1,p'}(\Omega) \coloneqq (W^{1,p}_0(\Omega))^*$ with $1/p + 1/p' = 1$. Given a right-hand side $F\in W^{-1,p'}(\Omega)$ and discrete subspaces $U_h\subset  W^{1,p'}_0(\Omega)$ and $V_h\subset W^{1,p}_0(\Omega)$, we approximate the solution $\fru \in W^{1,p'}_0(\Omega)$ to $B\fru = F$ by a minimizer 
\begin{align}\label{eq:discreteMinProb}
\fru_h \in \argmin_{u_h \in U_h} \, \lVert Bu_h - F \rVert_{V^*_h}.
\end{align}
For $p=2$ the computation of the minimizer in \eqref{eq:discreteMinProb} has been discussed in \cite[Sec.~3.2]{MonsuurStevensonStorn23}. For $p>2$ we modify the saddle point problem therein by introducing a computable weight $\sigma_n^\zeta$ with values within some relaxation interval $\zeta = [\zeta_-,\zeta_+]\subset (0,\infty)$ in the sense that $\sigma_n^\zeta(x) \in \zeta$ for almost all $x\in \Omega$. The resulting scheme seeks $\psi_{h,{n+1}}\in V_h$ and $\fru_{h,n+1} \in U_h$ with
\begin{align}\label{eq:Intro}
\begin{aligned}
\int_\Omega (\sigma_n^\zeta)^{2-p'} \nabla \psi_{h,n+1} \cdot \nabla v_h \dx + B\fru_{h,n+1}(v_h) &= F(v_h)&&\text{for all }v_h\in V_h,\\
Bu_h(\psi_{h,n+1}) & = 0&&\text{for all }u_h\in U_h.
\end{aligned}
\end{align}
Solving this problem allows us to update the weight and to proceed inductively.

To verify the convergence of the iterative scheme, we introduce in Section~\ref{sec:PrimalAndDual} equivalent formulations of the problem in \eqref{eq:discreteMinProb} using duality. 
Since the resulting problems share similarities with the $p$-Laplace problem, we can exploit recent ideas for the $p$-Laplace operator from \cite{DieningFornasierTomasiWank20,BalciDieningStorn22}. In particular, we introduce a regularization of the dual problem via a relaxation interval $\zeta = [\zeta_-,\zeta_+]$ and show convergence of the minimizers of the regularized problem towards the exact minimizer as $\zeta_- \to 0$ and $\zeta_+ \to \infty$ in Section~\ref{sec:regularization}. 
We verify the convergence of the \Kacanov{} iterations towards the minimizes of the regularized dual problems in Section~\ref{sec:KacanovScheme}. Additionally, we use duality again to rewrite the \Kacanov{} iterations as a primal problem, leading to the scheme in~\eqref{eq:Intro}. We conclude our analysis with a study of a priori and a posteriori error estimates in Section~\ref{sec:ErrorControl} and suggest an adaptive scheme in Section~\ref{sec:AdaptiveScheme}. Finally, we study numerically the beneficial properties of the scheme and discuss strategies to solve challenging problems like convection-dominated diffusion with vanishing viscosity in Section~\ref{sec:Applications}.
\section{Primal and Dual Formulation}\label{sec:PrimalAndDual}
Before we discuss the problem in \eqref{eq:discreteMinProb}, let us introduce some notation:
\begin{itemize}
\item The operator $B \colon W^{1,p'}_0(\Omega) \to W^{-1,p'}(\Omega)$ with exponent $p>2$ is a bounded linear mapping, defining a bilinear form $b(u,v) \coloneqq Bu(v)$ for all $u\in W^{1,p'}_0(\Omega)$ and $v \in W^{1,p}_0(\Omega)$. Moreover, let $F\in W^{-1,p'}(\Omega)$ be some given data.
\item Given a regular triangulation $\tria$ of the bounded Lipschitz domain $\Omega \subset \mathbb{R}^d$, set for all $\ell \in \mathbb{N}_0$ the space of piece-wise polynomials $\mathcal{L}^0_\ell(\tria) \coloneqq \lbrace w \in L^2(\Omega) \colon w|_T$ is a polynomial of maximal degree $\ell$ for all $T\in \tria\rbrace$ and set for some fixed degrees $k,\delta \in \mathbb{N}$ the Lagrange finite element spaces
\begin{align*}
U_h&\coloneqq \mathcal{L}^1_{k,0}(\tria) \coloneqq \mathcal{L}^0_{k}(\tria) \cap  W^{1,p'}_0(\Omega),\\
V_h&\coloneqq \mathcal{L}^1_{k+\delta,0}(\tria) \coloneqq \mathcal{L}^0_{k+\delta}(\tria) \cap  W^{1,p}_0(\Omega).
\end{align*}
\item Set for all $G\in W^{-1,p'}(\Omega)$ the discrete dual seminorm 
\begin{align*}
\lVert G \rVert_{V^*_h} \coloneqq \sup_{v_h\in V_h\setminus\lbrace 0 \rbrace} \frac{G(v_h)}{\lVert \nabla v_h \rVert_{L^p(\Omega)}}.
\end{align*}
\item Set the subspace $(BU_h)^\perp \coloneqq \lbrace v_h \in V_h\colon b(u_h ,v_h) = 0$ for all $u_h \in U_h\rbrace \subset V_h$.
\end{itemize}
We can characterize the solution $\fru_h \in U_h$ to \eqref{eq:discreteMinProb} via the saddle point problem: Seek $\psi_h\in V_h$ and $\fru_h \in U_h$ such that
\begin{align}\label{eq:Saddle1}
\begin{aligned}
\int_\Omega |\nabla \psi_h|^{p-2} \nabla \psi_h \cdot \nabla v_h \dx + b(\fru_h,v_h) &= F(v_h)&&\text{for all }v_h \in V_h,\\
b(u_h,\psi_h) &= 0 &&\text{for all }u_h \in U_h.
\end{aligned}
\end{align}
A further related problems seeks the minimizer
\begin{align}\label{eq:discreteMinProb2}
\psi_h = \argmin_{v_h \in (BU_h)^\perp} \frac{1}{p} \int_\Omega  |\nabla v_h |^p \dx - F(v_h).
\end{align}
\begin{lemma}[Existence and equivalent characterization]\label{lem:exAndEqui}
\ 
\begin{enumerate}
\item There exists a unique solution $\psi_h \in (BU_h)^\perp$ to the minimization problem in \eqref{eq:discreteMinProb2}.\label{itm:exPsiH}
\item There exists a solution $\fru_h \in U_h$ to the minimization problem in \eqref{eq:discreteMinProb}. The solution $\fru_h$ is unique up to the kernel $\ker B|_{U_h} \coloneqq \lbrace u_h \in U_h\colon Bu_h = 0$ in $V_h^*\rbrace$.\label{itm:exUh}
\item The pair $(\psi_h,\fru_h) \in V_h \times U_h$ solves \eqref{eq:Saddle1} if and only if $\fru_h \in U_h$ solves \eqref{eq:discreteMinProb} and $\psi_h \in (BU_h)^\perp$ solves \eqref{eq:discreteMinProb2}.\label{itm:equi}
\end{enumerate}
\end{lemma}
This lemma is shown in \cite[Thm.~4.1]{MugaZee20} within an abstract framework.
We give a direct proof utilizing the following statement.
\begin{lemma}[Duality mapping]\label{lem:dualityMapping}
Let $G\in V_h^*$.
\begin{enumerate}
\item There exists a unique solution $R(G) \in V_h$ to the problem\label{itm:sada}
\begin{align*}
\int_\Omega |\nabla R(G)|^{p-2} \nabla R(G) \cdot \nabla v_h \dx = G(v_h)\qquad\text{for all }v_h\in V_h.
\end{align*}
\item If $G \neq 0$, the function $\lVert \nabla R(G) \rVert_{L^p(\Omega)}^{-1} R(G)$ is the unique normed function that attains the supremum in the definition of the $V_h^*$ norm of $G$ in the sense that any $\Theta_h \in V_h \text{ with }\lVert \nabla \Theta_h \rVert_{L^p(\Omega)} = 1$ satisfies
\begin{align*}
G(\Theta_h) = \lVert G \rVert_{V_h^*}\quad\text{if and only if}\quad \Theta_h = \lVert \nabla R(G) \rVert_{L^p(\Omega)}^{-1} R(G).
\end{align*} 
\end{enumerate}
\end{lemma}
\begin{proof}
Let $G\in V_h^*$. 
The direct method in calculus of variations yields the existence of unique minimizers $R(G) \in V_h$ with
\begin{align*}
R(G) = \argmin_{v_h\in V_h} \frac{1}{p} \int_\Omega |\nabla v_h|^p \dx - G(v_h).
\end{align*}
Differentiation shows that this existence result is equivalent to the statement in \ref{itm:sada}.

Let $G\neq 0$.  H\"older's inequality and testing with $v_h = R(G)$ shows that 
\begin{align*}
\lVert G \rVert_{V_h^*} = \sup_{v_h\in V_h\setminus \lbrace 0 \rbrace} \frac{\int_\Omega |\nabla R(G)|^{p-2} \nabla R(G) \cdot \nabla v_h\dx}{\lVert \nabla v_h \rVert_{L^p(\Omega)}} = \lVert \nabla R(G)\rVert_{L^p(\Omega)}^{p-1}.
\end{align*}
This yields $G(\lVert  \nabla R(G)\rVert_{L^p(\Omega)}^{-1} R(G)) = \lVert G \rVert_{V_h^*}$.
Let $\Theta_h\in V_h$ with $\lVert \nabla \Theta_h \rVert_{L^p(\Omega)} = 1$ be a further function that attains the supremum in the sense that 
$
G(\Theta_h) = \lVert G \rVert_{V_h^*}. 
$
The linearity of $G$ implies 
\begin{align*}
 \sup_{v_h\in V_h\setminus \lbrace 0 \rbrace} \frac{G(v_h)}{\lVert \nabla v_h \rVert_{L^p(\Omega)}} = \lVert G \rVert_{V_h^*} = G\left( \tfrac{1}{2}\lVert \nabla R(G)\rVert_{L^p(\Omega)}^{-1} R(G) + \tfrac{1}{2} \Theta_h\right),
\end{align*}
which yields in particular that $1 \leq \big\lVert  \tfrac{1}{2}\lVert \nabla R(G)\rVert_{L^p(\Omega)}^{-1} \nabla R(G) + \tfrac{1}{2}\nabla \Theta_h\big\rVert_{L^p(\Omega)}$.
This estimate and the triangle inequality shows that 
\begin{align*}
2 &\leq \left\lVert \lVert \nabla R(G)\rVert_{L^p(\Omega)}^{-1}\nabla  R(G) + \nabla \Theta_h \right\rVert_{L^p(\Omega)}\\
& \leq \left\lVert \lVert \nabla R(G)\rVert_{L^p(\Omega)}^{-1} \nabla R(G) \right\rVert_{L^p(\Omega)} +\lVert\nabla \Theta_h\rVert_{L^p(\Omega)} = 2.
\end{align*} 
Since $W^{1,p}_0(\Omega)$ is a strictly convex space \cite{Hanner56}, this identity yields 
\begin{align*}
\Theta_h &= \lVert \nabla R(G)\rVert_{L^p(\Omega)}^{-1} R(G).\qedhere
\end{align*}
\end{proof}
\begin{proof}[Proof of Lemma~\ref{lem:exAndEqui}]
\textit{Step 1 (Proof of \ref{itm:exPsiH} and \ref{itm:exUh}).}
The direct method in calculus of variations yields the existence of unique minimizers $\psi_h \in (BU_h)^\perp$ of the strictly convex energy in \eqref{eq:discreteMinProb2}, that is, it verifies~\ref{itm:exPsiH}.
Similarly, we conclude the existence of a unique minimizer
\begin{align*}
B\fru_h = \argmin_{\lbrace Bu_h \colon u_h \in U_h\rbrace} \lVert Bu_h - F\rVert_{V_h^*}. 
\end{align*}
This yields the existence of a minimizer $\fru_h \in U_h$ to the problem \eqref{eq:discreteMinProb} and shows \ref{itm:exUh}. 

\textit{Step 2 (Proof of \ref{itm:equi}, trivial case).}
Let $\fru_h\in U_h$ and $\psi_h\in (BU_h)^\perp$ satisfy  \eqref{eq:discreteMinProb} and \eqref{eq:discreteMinProb2}. 
If $B\fru_h = F$ in   $V_h^*$, the problem in \eqref{eq:Saddle1} is satisfied with $\psi_h = 0$ and vice versa.

\textit{Step 3 (Proof of ``$\Leftarrow$'' in \ref{itm:equi}).}
Let $\fru_h\in U_h$ satisfy \eqref{eq:discreteMinProb} and let $\psi_h\in (BU_h)^\perp$ satisfy \eqref{eq:discreteMinProb2} with $B\fru_h \neq F$ in $V_h^*$. Since $\lbrace Bu_h\colon u_h \in U_h\rbrace$ is a closed subspace of $V_h^*$, a consequence of the Hahn-Banach theorem (see for example \cite[Prop.~3]{Zeidler95}) yields the existence of a function $\Theta_h \in V_h$ with $\lVert \nabla \Theta_h \rVert_{L^p(\Omega)} = 1$, 
\begin{align}\label{eq:ProofTeatsfsa}
(F - B \fru_h)(\Theta_h) = \lVert B \fru_h - F \rVert_{V_h^*},\quad\text{and}\quad B u_h (\Theta_h)  = 0\quad\text{for all }u_h \in U_h.
\end{align}
Lemma~\ref{lem:dualityMapping} characterizes the function $\Theta_h \in V_h$ due to the first identity in \eqref{eq:ProofTeatsfsa} as $\Theta_h = \lVert \phi_h \rVert_{L^p(\Omega)}^{-1} \phi_h$, where $\phi_h \in V_h$ solves the problem
\begin{align}\label{eq:asfvsca}
\int_\Omega |\nabla \phi_h|^{p-2} \nabla \phi_h \cdot \nabla v_h \dx = (F  -B\fru_h)(v_h) \qquad\text{for all }v_h \in V_h.
\end{align} 
In particular, the function $\phi_h$ solves 
\begin{align*}
\int_\Omega |\nabla \phi_h|^{p-2} \nabla \phi_h \cdot \nabla v_h \dx = F( v_h)\qquad\text{for all }v_h \in (BU_h)^\perp.
\end{align*}
Since this characterizes the minimizer in \eqref{eq:discreteMinProb2} and 
$\phi_h \in (BU_h)^\perp$ due to the second identity in \eqref{eq:ProofTeatsfsa}, we have $\phi_h = \psi_h$. Hence, the functions $(\fru_h,\psi_h) \in U_h\times V_h$ solve~\eqref{eq:Saddle1}.

\textit{Step 4 (Proof of ``$\Rightarrow$'' in \ref{itm:equi}).}
If there exists a solution $(\fru_h,\psi_h) \in U_h\times V_h$ to \eqref{eq:Saddle1}, the function $\psi_h\in V_h$ is an element in $(BU_h)^\perp$ and satisfies in particular
\begin{align}\label{eq:safgsadg}
\int_\Omega |\nabla \psi_h|^{p-2} \nabla \psi_h \cdot \nabla v_h \dx = F (v_h)\qquad\text{for all }v_h \in (BU_h)^\perp.
\end{align}
This identity characterizes the unique (Step 1) solution to \eqref{eq:discreteMinProb2}, that is, $\psi_h$ must be the minimizer in \eqref{eq:discreteMinProb2}.
The solution $\fru_h\in U_h$ to \eqref{eq:Saddle1} is characterized via the identity
\begin{align}\label{eq:zzzzzzzzzzzsda}
b(\fru_h,v_h) = F(v_h) - \int_\Omega |\nabla \psi_h|^{p-2} \nabla \psi_h\cdot v_h\dx \qquad\text{for all }v_h \in V_h.
\end{align}
Since the right-hand side equals zero for all $v_h \in (BU_h)^\perp$ due to~\eqref{eq:safgsadg}, it is in the range of the operator $B\colon U_h\to V_h^*$, that is, there exist a unique solution $\fru_h \in U_h/\ker B|_{U_h}$ to \eqref{eq:zzzzzzzzzzzsda}. 
We know from Step 3 that the solution to \eqref{eq:discreteMinProb} solves the problem in \eqref{eq:zzzzzzzzzzzsda} as well. The uniqueness of these solutions up to the kernel $\ker B|_{U_h}$ (Step 1) implies that they must coincide.
\end{proof}
The minimization problem in \eqref{eq:discreteMinProb2} shares similarities with the $p$-Laplace problem, which can be solved by the regularized \Kacanov{} scheme introduced in \cite{DieningFornasierTomasiWank20}. Unfortunately, this schemes converges only for $p \leq 2$. We remedy this downside as in \cite{BalciDieningStorn22} by the use of duality. The dual problem of \eqref{eq:discreteMinProb2} involves the affine space 
\begin{align}\label{eq:DefSigma}
\Sigma \coloneqq \left\lbrace \tau \in L^{p'}(\Omega;\mathbb{R}^d) \colon \int_\Omega \nabla v_h \cdot \tau\dx  = F(v_h) \text{ for all } v_h \in (BU_h)^\perp \right\rbrace.
\end{align}
It seeks the minimizer to the problem 
\begin{align}\label{eq:dualProb}
\sigma = \argmin_{\tau \in \Sigma} \frac{1}{p'} \int_\Omega |\tau|^{p'}\dx. 
\end{align}
Let us show the equivalence of the problems in \eqref{eq:discreteMinProb2} and \eqref{eq:dualProb}. The solution to \eqref{eq:discreteMinProb2} is characterized via the Euler-Lagrange equation as unique solution $\psi_h \in (BU_h)^\perp$ to
\begin{align}\label{eq:EulerLagrange}
\int_\Omega |\nabla \psi_h|^{p-2} \nabla\psi_h \cdot \nabla v_h \dx = F(v_h)\qquad\text{for all } v_h\in (BU_h)^\perp.
\end{align}
The solution $\sigma \in L^{p'}(\Omega;\mathbb{R}^d)$ to \eqref{eq:dualProb} solves with unique function $\phi_h \in (BU_h)^\perp$ the saddle point problem
\begin{align}\label{eq:EulerLagrange2}
\begin{aligned}
\int_\Omega |\sigma|^{p'-2} \sigma \cdot \tau \dx - \int_\Omega \nabla \phi_h \cdot \tau \dx &= 0 &&\text{for all } \tau \in L^{p'}(\Omega;\mathbb{R}^d), \\
- \int_\Omega \nabla v_h \cdot \sigma \dx & = - F(v) &&\text{for all } v_h\in (BU_h)^\perp.
\end{aligned}
\end{align}
\begin{lemma}[Duality]\label{lem:Duality}
The solutions to \eqref{eq:EulerLagrange} and \eqref{eq:EulerLagrange2} are related via the identities 
\begin{align}\label{eq:Ident}
\sigma = |\nabla \psi_h|^{p-2} \nabla \psi_h,\qquad \nabla \psi_h = |\sigma|^{p'-2} \sigma,\qquad\text{and}\qquad \phi_h = \psi_h.
\end{align}
Furthermore, the minimal energies satisfy 
\begin{align}\label{eq:equalEnergies}
\frac{1}{p} \int_\Omega |\nabla \psi_h|^p \dx - F(\psi_h) = - \frac{1}{p'} \int_\Omega |\sigma|^{p'}\dx. 
\end{align}
\end{lemma}
\begin{proof}
Let $\phi_h \in (BU_h)^\perp$ solve \eqref{eq:discreteMinProb2} and define the functions 
\begin{align*}
\sigma \coloneqq  |\nabla \psi_h|^{p-2} \nabla \psi_h \in L^{p'}(\Omega;\mathbb{R}^d)\qquad\text{and}\qquad \phi_h \coloneqq \psi_h \in (BU_h)^\perp. 
\end{align*}
Direct calculations show that these functions solve the saddle point problem in~\eqref{eq:EulerLagrange2}.
 Since the solution to \eqref{eq:EulerLagrange2} is unique (due to the uniqueness of the minimizer $\sigma$ and the fact that the first line in  \eqref{eq:EulerLagrange2} uniquely determines $\phi_h$ via the identity $|\sigma|^{p'-2} \sigma = \nabla \phi_h$), we obtain the equivalence stated in \eqref{eq:Ident}. 
Since $1/p+1/p'=1$ implies with~\eqref{eq:Ident} that $|\nabla \psi_h|^p = |\sigma|^{p'}$, the identity in~\eqref{eq:EulerLagrange} yields 
\begin{align*}
\frac{1}{p} \int_\Omega |\nabla \psi_h|^p \dx - F(\psi_h)  = \left(\frac{1}{p}-1\right) \int_\Omega |\nabla \psi_h|^p \dx = -\frac{1}{p'} \int |\sigma|^{p'}\dx.
\end{align*}
This shows \eqref{eq:equalEnergies} and concludes the proof.
\end{proof}
We want to solve the non-linear problem in \eqref{eq:EulerLagrange2} via the iterative scheme
\begin{align*}
\begin{aligned}
\int_\Omega |\sigma_n|^{p'-2} \sigma_{n+1} \cdot \tau \dx - \int_\Omega \nabla \phi_{h,n+1} \cdot \tau \dx &= 0 &&\text{for all } \tau \in L^{p'}(\Omega;\mathbb{R}^d), \\
- \int_\Omega \nabla v_h \cdot \sigma_{n+1} \dx & = - F(v_h) &&\text{for all } v_h\in (BU_h)^\perp.
\end{aligned}
\end{align*}
However, the resulting problems are in general not well posed since $\sigma_n$ might degenerate. 
We thus introduce the following regularization.
\section{Regularization}\label{sec:regularization}
Following \cite{DieningFornasierTomasiWank20} and \cite{BalciDieningStorn22}, we define for any relaxation interval $\zeta = [\zeta_-,\zeta_+] \subset (0,\infty)$ and all $t \geq 0$ the integrant 
\begin{align*}
\kappa^*_\zeta(t) \coloneqq \begin{cases}
\frac{1}{2}\zeta_-^{p'-2}t^2 + \left(\frac{1}{p'}-\frac{1}{2}\right)\zeta_-^{p'}&\text{for }t\leq \zeta_-,\\
\frac{1}{p'}t^{p'}&\text{for } \zeta_-\leq t\leq \zeta_+,\\
\frac{1}{2}\zeta_+^{p'-2}t^2+\left(\frac{1}{p'}-\frac{1}{2}\right)\zeta_+^{p'}&\text{for }\zeta_+ \leq t.
\end{cases}
\end{align*}
We furthermore define for all $\tau \in L^{p'}(\Omega;\mathbb{R}^d)$ the energies
\begin{align*}
\mathcal{J}^*_\zeta(\tau) \coloneqq \int_\Omega \kappa^*_\zeta(|\tau|)\dx \qquad\text{and}\qquad 
\mathcal{J}^*(\tau)\coloneqq \frac{1}{p'} \int_\Omega |\tau|^{p'} \dx.
\end{align*}
Notice that the regularized energy $\mathcal{J}^*_\zeta(\tau_h)$  equals infinity if $\tau \in L^{p'}(\Omega;\mathbb{R}^d) \setminus L^2(\Omega;\mathbb{R}^d)$. Furthermore, the relaxed energy is monotone with respect to the relaxation interval in the sense that all $\tau \in L^{p'}(\Omega;\mathbb{R}^d)$ and relaxation intervals $\zeta^2 = [\zeta^2_-,\zeta^2_+] \subset \zeta^1 = [\zeta^1_-,\zeta^1_+] \subset (0,\infty)$ satisfy
\begin{align*}
\mathcal{J}^*(\tau) \leq \mathcal{J}^*_{\zeta^1}(\tau)\leq \mathcal{J}^*_{\zeta^2}(\tau).
\end{align*}
The direct method in calculus of variations verifies the existence of a unique minimizer $\sigma_\zeta$ of $\mathcal{J}_\zeta^*$ in $\Sigma$ in the sense that
\begin{align}\label{eq:assumptionExistence}
\sigma_\zeta = \argmin_{\tau \in \Sigma} \mathcal{J}_\zeta^*(\tau_h).
\end{align}
In the following we investigate the convergence of $\sigma_\zeta$ towards the minimizer $\sigma\in \Sigma$ in~\eqref{eq:dualProb}. Rather than investigating convergence in the $L^p(\Omega)$ norm, we investigate the convergence of the energies. This energy difference leads to the following bound.
\begin{lemma}[Notion of distance]\label{lem:NotionOfDistance}
Let $\sigma\in \Sigma$ be the minimizer in~\eqref{eq:dualProb} and let $\tau \in \Sigma$. Then we have 
\begin{align*}
\lVert |\sigma| + |\sigma - \tau| \rVert_{L^{p'}(\Omega)}^{p'-2}
\lVert \sigma - \tau \rVert_{L^{p'}(\Omega)}^2 & \lesssim \mathcal{J}^*(\tau) - \mathcal{J}^*(\sigma) \lesssim \lVert \sigma - \tau \rVert_{L^{p'}(\Omega)}^{p'}.
\end{align*}
Furthermore, we have the lower bound 
\begin{align*}
\lVert |\tau| + |\sigma - \tau| \rVert_{L^{p'}(\Omega)}^{p'-2}
\lVert \sigma - \tau \rVert_{L^{p'}(\Omega)}^2 \lesssim \mathcal{J}^*(\tau) - \mathcal{J}^*(\sigma).
\end{align*}
The hidden constants depend on $p$ but are independent of the solution $\sigma$.
\end{lemma}
\begin{proof}
Since this result is well-known in the context of the $p$-Laplacian, let us briefly summarize its derivation.
Let $\sigma$ and $\tau$ be as in the lemma. Since $(\mathcal{J}^*)'(\sigma)(\tau - \sigma) = 0$ due to the minimization property of $\sigma$, the convexity of $\mathcal{J}^*$ yields
\begin{align*}
\mathcal{J}^*(\tau) - \mathcal{J}^*(\sigma)& \leq (\mathcal{J}^*)'(\tau) (\tau - \sigma) = \big((\mathcal{J}^*)'(\tau) - (\mathcal{J}^*)'(\sigma)\big) (\tau - \sigma)\\
& = \int_\Omega (|\tau|^{p'-2} \tau - |\sigma|^{p'-2} \sigma) \cdot (\tau - \sigma) \dx.
\end{align*}
Further arguments for the integrand as for example shown in \cite[Lem.~42]{DieningFornasierTomasiWank20} lead to the lower bound
\begin{align*}
\int_\Omega (|\tau|^{p'-2} \tau - |\sigma|^{p'-2} \sigma) \cdot (\tau - \sigma) \dx\lesssim \mathcal{J}^*(\tau) - \mathcal{J}^*(\sigma).
\end{align*}
Additionally, the equivalence $(|P|^{p'-2}P - |Q|^{p'-2}Q) \cdot (P-Q) \eqsim (|Q| + |P-Q|)^{p'-2} |P-Q|^2$ for all  $P,Q\in \mathbb{R}^d$ as shown in \cite[Lem.~39]{DieningFornasierTomasiWank20} implies  
\begin{align}\label{eq:asfgsra}
\int_\Omega (|\tau|^{p'-2} \tau - |\sigma|^{p'-2} \sigma) \cdot (\tau - \sigma) \dx \eqsim \int_\Omega (|\sigma| + |\sigma - \tau|)^{p'-2} |\sigma - \tau|^2\dx.
\end{align}
These observations lead to the upper bound in the lemma. The lower bound follows from H\"older's reverse inequality
\begin{align*}
&\left(\int_\Omega \big((|\sigma| + |\sigma - \tau|)^{p'-2}\big)^{\tfrac{1}{1-q}} \dx\right)^{1-q} \left( \int_\Omega \big(|\sigma - \tau|^2\big)^{\tfrac{1}{q}} \dx\right)^q\\
&\qquad\qquad\qquad \leq \int_\Omega (|\sigma| + |\sigma - \tau|)^{p'-2} |\sigma - \tau|^2\dx\qquad\text{with }q \coloneqq \frac{2}{p'}.
\end{align*}
Exchanging the role of $\sigma$ and $\tau$ in \eqref{eq:asfgsra} leads to the alternative lower bound.
\end{proof}
We have the following convergence result for the energy differences. 
\begin{proposition}[Convergence in $\zeta$]\label{prop:convInZeta}
Let $\psi_h \in V_h$ denote the solution to \eqref{eq:discreteMinProb2} and let $\sigma$ and $\sigma_\zeta$ denote the minimizers in \eqref{eq:dualProb} and \eqref{eq:assumptionExistence}, respectively.
Their energy difference is bounded for all relaxation intervals $\zeta = [\zeta_-,\zeta_+]\subset (0,\infty)$ and $r > 2$ by
\begin{align*}
 \mathcal{J}^*(\sigma_\zeta) - \mathcal{J}^*(\sigma)
\leq \mathcal{J}_\zeta^*(\sigma_\zeta) - \mathcal{J}^*(\sigma) &\leq \frac{|\Omega|}{p'}\zeta_-^{p'} + \frac{1}{p'}\zeta_+^{-(r-p')} \lVert  \sigma \rVert^{r}_{L^{r}(\Omega)} \\
 & = \frac{|\Omega|}{p'}\zeta_-^{p'} + \frac{1}{p'}\zeta_+^{-(r-p')} \lVert \nabla \psi_h \rVert^{r(p-1)}_{L^{r(p-1)}(\Omega)}.
\end{align*}
\end{proposition}
\begin{proof}
This first two inequalities follow as in~\cite[Thm.~3.1]{BalciDieningStorn22}. 
Since $|\nabla \psi_h|^p = |\sigma|^{p'}$ due to Lemma~\ref{lem:Duality}, the equality then follows from the identity 
\begin{align*}
&\int_\Omega |\sigma|^r \dx = \int_\Omega |\sigma|^{p' \tfrac{r}{p'}} = \int_\Omega |\nabla \psi |^{r \tfrac{p}{p'}} \dx = \int_\Omega |\nabla \psi_h |^{r(p-1)} \dx.\qedhere
\end{align*}
\end{proof}
\begin{remark}[Regularity]
The convergence result in Proposition~\ref{prop:convInZeta} assumes the regularity property $\psi_h \in W_0^{1,r(p-1)}(\Omega)$. Such a result is indeed  true for all $r \leq \infty$, since $\psi_h \in V_h= \mathcal{L}^1_{k+\delta ,0}(\tria)$ is a function in a finite dimensional space. However, the norm might increase as the mesh is refined. In practical computations this issue does not seem to cause problems, since we can control the impact of the regularization by comparing the energies $\mathcal{J}_\zeta^*(\tau)$ and $\mathcal{J}^*(\tau)$, cf.~Section~\ref{sec:AdaptiveScheme}, and our numerical experiments in Section~\ref{sec:Applications} do not indicate a significantly decreased rate of convergence. 
\end{remark}
\section{Relaxed \Kacanov{} scheme}\label{sec:KacanovScheme}
In this section we introduce an iterative scheme that converges towards the minimizer $\sigma_\zeta$ in \eqref{eq:assumptionExistence} with relaxation interval $\zeta = [\zeta_-,\zeta_+]\subset (0,\infty)$. Set $b \vee c \coloneqq \max \lbrace b,c\rbrace$ and $b\wedge c\coloneqq \min\lbrace b,c\rbrace$ for all $b,c\in\mathbb{R}$. Given some initial value $\sigma_0\in L^{p'}(\Omega;\mathbb{R}^d)$, we compute iteratively for any $n\in \mathbb{N}_0$ the solution $\sigma_{n+1} \in L^{p'}(\Omega;\mathbb{R}^d)$ and $\psi_{h,n+1}\in (BU_h)^\perp$ satisfying for all $\tau \in L^{p'}(\Omega;\mathbb{R}^d)$ and $v_h\in (BU_h)^\perp$
\begin{align}\label{eq:KacanovRelaxed}
\begin{aligned}
\int_\Omega (\zeta_-\vee |\sigma_n|\wedge \zeta_+) ^{p'-2} \sigma_{n+1} \cdot \tau \dx - \int_\Omega \nabla \psi_{h,n+1} \cdot \tau \dx &= 0,\\
- \int_\Omega \nabla v_h \cdot \sigma_{n+1} \dx & = - F(v_h).
\end{aligned}
\end{align}
The following proposition shows convergence of the solutions $\sigma_n$ towards the minimizer $\sigma_\zeta$ in~\eqref{eq:assumptionExistence}.
\begin{proposition}[Convergence]\label{prop:linConvergence}
There exists a constant $\rho \lesssim (\zeta_-/\zeta_+)^{2-p'}$ such that
\begin{align*}
\rho \big(\mathcal{J}_\zeta^*(\sigma_{n}) - \mathcal{J}_\zeta^*(\sigma_\zeta)\big) \leq \mathcal{J}_\zeta^*(\sigma_n) - \mathcal{J}_\zeta^*(\sigma_{n+1}) \qquad\text{for all }n\in \mathbb{N}.
\end{align*}
Moreover, we have the convergence result
\begin{align*}
\mathcal{J}_\zeta^*(\sigma_{n+1}) - \mathcal{J}_\zeta^*(\sigma_\zeta) \leq (1-\rho)^n \big( \mathcal{J}_\zeta^*(\sigma_0) - \mathcal{J}_\zeta^*(\sigma_\zeta) \big)\qquad\text{for all }n\in \mathbb{N}.
\end{align*}
\end{proposition}
\begin{proof}
This result follows as in \cite[Sec.~4]{BalciDieningStorn22}.
\end{proof}
To solve the problem in \eqref{eq:KacanovRelaxed}, we utilize duality to obtain a primal problem which seeks $\psi_{h,n+1} \in (BU_h)^\perp$ such that for all $v_h\in (BU_h)^\perp$
\begin{align}\label{eq:KacanovPrimal}
\int_\Omega (\zeta_-\vee |\sigma_n|\wedge \zeta_+)^{2-p'} \nabla \psi_{h,n+1} \cdot \nabla v_h \dx = F(v_h).
\end{align}
The corresponding saddle point problem seeks $\psi_{h,n+1}\in V_h$ and $\fru_{h,n+1}\in U_h$ with
\begin{align*}
\begin{aligned}
\int_\Omega (\zeta_-\vee |\sigma_n|\wedge \zeta_+)^{2-p'} \nabla \psi_{h,n+1} \cdot \nabla v_h \dx + b(\fru_{h,n+1},v_h) &= F(v_h)&&\text{for all }v_h\in V_h,\\
b(u_h,\psi_{h,n+1}) & = 0&&\text{for all }u_h\in U_h.
\end{aligned}
\end{align*}
\begin{proposition}[Equivalence]
The solution $\psi_{h,n+1} \in (BU_h)^\perp$ to \eqref{eq:KacanovPrimal} and $\sigma_{n+1}\in \Sigma$ to \eqref{eq:KacanovRelaxed} are related via the identity
\begin{align*}
\sigma_{n+1} = (\zeta_-\vee |\sigma_n|\wedge \zeta_+)^{2-p'} \nabla \psi_{h,n+1}.
\end{align*}
\end{proposition}
\begin{proof}
The same arguments as in the proof of Lemma~\ref{lem:Duality} yield the proposition.
\end{proof}
The problem in \eqref{eq:KacanovPrimal} can be solved iteratively, leading to a convergent scheme. Adaptivity, as discussed in Section~\ref{sec:AdaptiveScheme} below, might improve the convergence.
\section{Error control}\label{sec:ErrorControl}
The a priori and a posteriori error control for minimal residual methods is well established, see for example \cite{CarstensenDemkowiczGopalakrishnan14,CarstensenDemkowiczGopalakrishnan16,MugaZee20,Storn20,MonsuurStevensonStorn23}.
Let us briefly adapt the proofs therein to our situation. We assume that
\begin{enumerate}
\item there exists a unique solution $\fru \in W^{1,p'}_0(\Omega)$ with $B\fru = F$ in $W^{-1,p'}(\Omega)$ and\label{itm:Ass1}
\item there exist a Fortin operator $\Pi \colon W^{1,p}_0(\Omega) \to V_h$ with continuity constant $\lVert \Pi \rVert< \infty $ in the sense that for all $u_h\in U_h$ and $v\in W^{1,p}_0(\Omega)$ \label{itm:Ass3}
\begin{align}\label{eq:Fortin}
b(u_h,v-\Pi v) = 0 \quad \text{and}\quad \lVert \nabla \Pi v \rVert_{L^p(\Omega)} \leq \lVert \Pi \rVert\, \lVert  \nabla v \rVert_{L^p(\Omega)}.
\end{align}
\end{enumerate}
\begin{proposition}[Error control]\label{prop:errorControl}
Suppose that the assumptions in \ref{itm:Ass1} and \ref{itm:Ass3} are true. 
Then the solution $\fru$ to $B\fru = F$ in $W^{-1,p'}(\Omega)$ and $\fru_h\in U_h$ to \eqref{eq:discreteMinProb} satisfy
\begin{align}\label{eq:aPriori1}
\lVert B \fru - B\fru_h \rVert_{W^{-1,p'}(\Omega)} \leq (1+2\,\lVert \Pi \rVert) \min_{u_h\in U_h} \lVert B \fru - B u_h \rVert_{W^{-1,p'}(\Omega)}.
\end{align}
Moreover, with oscillation $\textup{osc}(F) \coloneqq \sup_{v\in W^{1,p}(\Omega)\setminus \lbrace 0 \rbrace} F(v-\Pi v)/\lVert \nabla v \rVert_{L^p(\Omega)}$ we have for any $u_h \in U_h$ the a posteriori error estimate
\begin{align}\label{eq:aPosteriori1}
\lVert B \fru - B u_h \rVert_{W^{-1,p'}(\Omega)} \leq \lVert \Pi \rVert\, \lVert B \fru - B u_h \rVert_{V_h^*} +  \textup{osc}(F).
\end{align}
\end{proposition}
\begin{proof}
Let $u_h \in U_h$. Any $v \in W^{1,p}_0(\Omega)$ with $\lVert \nabla v \rVert_{L^p(\Omega)} = 1$ satisfies 
\begin{align*}
b( \fru - u_h , v) = b( \fru - u_h , \Pi v) + b( \fru - u_h ,v - \Pi v) \leq \lVert \Pi\rVert\, \lVert B\fru -B u_h \rVert_{V^*_h} + \textup{osc}(F).
\end{align*}
This proves the a posteriori estimate in \eqref{eq:aPosteriori1}. To obtain the a priori estimate in \eqref{eq:aPriori1}, we use the minimization property in \eqref{eq:discreteMinProb}, that is,
\begin{align*}
\lVert B\fru -B \fru_h \rVert_{V^*_h} = \min_{u_h \in U_h} \lVert B\fru -B u_h \rVert_{V^*_h} \leq \min_{u_h \in U_h} \lVert B\fru -B u_h \rVert_{W^{-1,p'}(\Omega)}.
\end{align*}
Moreover, the oscillation satisfies 
\begin{align*}
\textup{osc}(F) &= \sup_{v\in W^{1,p}(\Omega)\setminus \lbrace 0 \rbrace} \frac{F(v-\Pi v)}{\lVert \nabla v \rVert_{L^p(\Omega)}} = \min_{u_h\in U_h} \sup_{v\in W^{1,p}(\Omega)\setminus \lbrace 0 \rbrace} \frac{(B\fru - B \fru_h)(v-\Pi v)}{\lVert \nabla v \rVert_{L^p(\Omega)}}\\
& \leq (1 + \lVert \Pi \rVert)\, \lVert B\fru - B \fru_h \rVert_{W^{-1,p'}(\Omega)}.
\end{align*}
Combining these estimates with \eqref{eq:aPosteriori1} concludes the proof of \eqref{eq:aPriori1}.
\end{proof}
\begin{corollary}[A posteriori for exact solution]\label{cor:aposteriori}
Let $(\psi_h,\fru_h) \in V_h\times U_h$ solve \eqref{eq:Saddle1} and assume that the assumptions in \ref{itm:Ass1} and \ref{itm:Ass3} are satisfied.
Moreover, let $\sigma\in \Sigma$ denote the solution to \eqref{eq:dualProb}. Then we have the a posteriori error estimate
\begin{align}\label{eq:aposterior2}
\lVert B\fru - B\fru_h \rVert_{W^{-1,p'}(\Omega)} \eqsim \lVert \nabla \psi_h \rVert^{p-1}_{L^p(\Omega)} + \textup{osc}(F)  = \lVert \sigma \rVert_{L^{p'}(\Omega)} + \textup{osc}(F).
\end{align}
\end{corollary}
\begin{proof}
H\"older's inequality, the first equation in \eqref{eq:Saddle1}, and the identity in \eqref{eq:Ident} show
\begin{align*}
\lVert B\fru - B\fru_h \rVert_{V_h^*} = \sup_{v_h \in V_h\setminus \lbrace 0 \rbrace} \frac{\int_\Omega |\nabla \psi_h|^{p-2}\nabla \psi_h \cdot \nabla v_h \dx }{\lVert \nabla v_h \rVert_{L^p(\Omega)}} = \lVert \nabla \psi_h \rVert^{p-1}_{L^p(\Omega)} = \lVert \sigma \rVert_{L^{p'}(\Omega)}.
\end{align*}
Using this identity in the a posteriori estimate in \eqref{eq:aPosteriori1} leads to the upper bound in \eqref{eq:aposterior2}.
Equivalence follows from the upper bound for the oscillation
\begin{align*}
\textup{osc}(F)& = \sup_{v \in W^{1,p}_0(\Omega)} \frac{F(v-\Pi v)}{\lVert \nabla v \rVert_{L^{p}(\Omega)}} = \sup_{v \in W^{1,p}_0(\Omega)} \frac{b(\fru - \fru_h,v-\Pi v)}{\lVert \nabla v \rVert_{L^{p}(\Omega)}}\\
& \leq (1 + \lVert \Pi \rVert)\, \lVert B\fru - B \fru_h \rVert_{W^{-1,p'}(\Omega)}.\qedhere
\end{align*}
\end{proof}
We conclude this section with a discussion of the following additional assumption:
\begin{enumerate}\setcounter{enumi}{2}
\item The operator $B\colon W^{1,p'}_0(\Omega) \to W^{-1,p'}(\Omega)$ is bounded from above and below in the sense that\label{itm:Ass2} 
\begin{align*}
\lVert \nabla u \rVert_{L^{p'}(\Omega)} \eqsim \lVert Bu\rVert_{W^{-1,p'}(\Omega)}\qquad\text{for all }u \in W^{1,p'}_0(\Omega). 
\end{align*}
\end{enumerate}
Under this additional assumption the error estimates in Proposition~\ref{prop:errorControl} and Corollary~\ref{cor:aposteriori} allow for any estimate of the more natural error quantity $\lVert \nabla \fru - \nabla \fru_h \rVert_{L^{p'}(\Omega)}$ due to the equivalence 
\begin{align*}
\lVert \nabla \fru - \nabla \fru_h \rVert_{L^{p'}(\Omega)} \eqsim \lVert B \fru  - B \fru_h \rVert_{W^{-1,p'}(\Omega)}.
\end{align*}
The assumption in \ref{itm:Ass1} seems to be natural. The assumption in \ref{itm:Ass3} can in many situations be achieved by choosing sufficiently large polynomial degrees $k+\delta$ for the test space $V_h = \mathcal{L}^1_{k+\delta}(\tria)$ as for example investigated in \cite[Sec.~4]{MonsuurStevensonStorn23}. 
The assumption in \ref{itm:Ass2} has been investigated in \cite{HoustonMugaRoggendorfvanderZee19} but seems to be rather restrictive.
Indeed, there exist counterexamples for the Laplace problem $Bu = \int_\Omega \nabla  u \cdot \nabla \bigcdot \dx$ for exponents $p>4$ and non-smooth non-convex domains $\Omega$ as shown in \cite{JerisonKenig95}. 
Notice that even in cases where \ref{itm:Ass2} is satisfied, the Galerkin scheme investigated in \cite{HoustonMugaRoggendorfvanderZee19} requires stability of the $W^{1,2}_0(\Omega)$-projection in $W^{1,p}_0(\Omega)$. Such stability results are known for uniform and mildly graded meshes \cite{DemlowLeykekhmanSchatzWahlbin12,DieningRolfesSalgado23}, but are an open problem for adaptively refined meshes. Our minimal residual method circumvents this problem by suitable designs of Fortin operators in \ref{itm:Ass3}. 
\section{Adaptive Scheme}\label{sec:AdaptiveScheme}
As pointed out in Corollary~\ref{cor:aposteriori}, the minimizer $\sigma$ with \eqref{eq:dualProb} allows us to drive an adaptive mesh refinement scheme. However, our iterative scheme does not compute the exact solution $\sigma$. 
We thus introduce an adaptive scheme that additionally takes the distance of the current iterate $\sigma_{n}$ to $\sigma$ into account.
The error indicator that indicate errors caused by
\begin{enumerate}
\item the upper interval bound $\zeta_+$ reads
$\eta_{\zeta_+}^2(\sigma_{n}) \coloneqq \mathcal{J}^*_{\zeta}(\sigma_{n}) - \mathcal{J}^*_{[\zeta_-,\infty)}(\sigma_{n})$,\label{item:1}
\item the lower interval bound $\zeta_-$ reads
$
\eta_{\zeta_-}^2(\sigma_{n}) \coloneqq \mathcal{J}^*_{\zeta}(\sigma_{n}) - \mathcal{J}^*_{[0,\zeta_+]}(\sigma_{n})$,\label{item:2}
\item the error due to the fixed-point iteration reads \label{item:3}
\begin{align*}
\eta^2_{\textup{Ka\v{c}},\zeta}(\sigma_{n}) \coloneqq \left(\frac{\zeta_+}{\zeta_-}\right)^{2-p'}  \big( \mathcal{J}^*_\zeta)(\sigma_n) - \mathcal{J}^*_\zeta(\sigma_{n+1})\big),
\end{align*}
\item the error due to the discretization reads\label{item:4}
\begin{align*}
\eta_h^{p'} \coloneqq \sum_{T\in \tria} \eta^{p'}_h(T)\quad \text{with}\quad \eta^{p'}_h(T) \coloneqq \lVert \sigma_{n} \rVert_{L^{p'}(T)}^{p'}.
\end{align*}
\end{enumerate}
The indicators in \ref{item:1} and \ref{item:2} provide some information on the impact of the relaxation interval $\zeta$ on the current iterate.
The indicator in \ref{item:3} is motivated by the convergence result in Proposition~\ref{prop:linConvergence}.
The error indicator in \ref{item:4} is motivated by the a posteriori error estimate for $\sigma$ in Corollary~\ref{cor:aposteriori}.
Notice that $\sigma_n$ is indeed a good approximation of $\sigma$ if $\lVert \sigma - \sigma_n \rVert_{L^{p'}(\Omega)} \ll \lVert \sigma \rVert_{L^{p'}(\Omega)}$, which can be seen by the triangle inequality 
\begin{align*}
|\lVert \sigma_n \rVert_{L^{p'}(\Omega)} - \lVert \sigma - \sigma_n \rVert_{L^{p'}(\Omega)} | \leq
\lVert \sigma \rVert_{L^{p'}(\Omega)} \leq \lVert \sigma_n \rVert_{L^{p'}(\Omega)} + \lVert \sigma - \sigma_n \rVert_{L^{p'}(\Omega)}.
\end{align*}
Lemma~\ref{lem:NotionOfDistance} states that
\begin{align*}
\lVert \sigma - \sigma_n \rVert^2_{L^{p'}(\Omega)} \lesssim \lVert |\sigma_n| + |\sigma - \sigma_n|\rVert_{L^{p'}(\Omega)}^{2-p'} \big(\mathcal{J}(\sigma_n) - \mathcal{J}(\sigma) \big).
\end{align*}
Hence, the estimate $\lVert \sigma - \sigma_n \rVert_{L^{p'}(\Omega)} \ll \lVert \sigma \rVert_{L^{p'}(\Omega)}$ follows from an estimate like
\begin{align*}
\mathcal{J}(\sigma_n) - \mathcal{J}(\sigma) \ll \lVert |\sigma_n| + |\sigma - \sigma_n|\rVert_{L^{p'}(\Omega)}^{p'-2}  \lVert \sigma_n\rVert_{L^{p'}(\Omega)}^2 \leq \lVert \sigma_n\rVert_{L^{p'}(\Omega)}^{p'} = \eta_h^{p'}.
\end{align*}
This motivates the following refinement strategy with some small weight $w>0$:
\begin{enumerate}
\item If $\eta_{\zeta_+}^2(\sigma_{n})  + \eta_{\zeta_-}^2(\sigma_{n})  + \eta^2_{\textup{Ka\v{c}},\zeta}(\sigma_{n}) \leq w \, \eta_h^{p'}$, refine the mesh adaptively with the local error contributions $\eta^{p'}_h(T)$ as refinement indicator.
\item Otherwise, if $\max \lbrace \eta_{\zeta_-}^2(\sigma_{n}),\eta^2_{\textup{Ka\v{c}},\zeta}(\sigma_{n})\rbrace \leq \eta_{\zeta_+}^2(\sigma_{n})$, increase $\zeta_+$.
\item Otherwise, if $\max \lbrace \eta_{\zeta_+}^2(\sigma_{n}),\eta^2_{\textup{Ka\v{c}},\zeta}(\sigma_{n})\rbrace\leq \eta_{\zeta_-}^2(\sigma_{n})$, decrease $\zeta_-$.
\end{enumerate}
Then we perform another \Kacanov{} iteration and continue with the evaluation of the resulting error indicators. This leads to an adaptive loop.
\begin{remark}[Primal-dual error estimator]
In \cite[Sec.~6.2]{BalciDieningStorn22} we use the dual problem with energy $\mathcal{J}_\zeta$ of the minimization problem in \eqref{eq:dualProb} to define the estimator 
\begin{align}\label{eq:Bds}
\eta^2_{\textup{Ka\v{c}},\zeta,\textup{Dual}}(\sigma_{n}) \coloneqq \mathcal{J}_\zeta(\psi_{h,n}) + \mathcal{J}^*_\zeta(\sigma_{n}).
\end{align}
This error estimator is a guaranteed upper bound for the error 
\begin{align*}
\mathcal{J}^*_\zeta(\sigma_n) - \mathcal{J}^*_\zeta(\sigma_\zeta) \leq \eta^2_{\textup{Ka\v{c}},\zeta,\textup{Dual}}(\sigma_{n}).
\end{align*}
However, in \cite{BalciDieningStorn22} we focused on a lowest-order scheme in the sense that $V_h = \mathcal{L}^1_{1,0}(\Omega)$, which allows for accurate evaluations of $\sigma_{n+1} = (\zeta_-\vee |\sigma_n|\wedge \zeta_+)^{2-p'} \nabla \psi_{h,n+1}$ with $\sigma_n\in \mathcal{L}^0_0(\tria;\mathbb{R}^d)$. Since in this paper's minimal residual method the space $V_h$ is of higher polynomial degree, the evaluation of $\sigma_{n+1}$ becomes more intricate. Our numerical experiments indicate that this challenge does not impact the convergence of the \Kacanov{} scheme, but it causes difficulties when evaluating the duality gap in~\eqref{eq:Bds}. We thus use the alternative indicator in \ref{item:3}. It is possible to circumvent these issues by replacing the test space $V_h$ with of higher polynomial degree by a test space $V_h =\mathcal{L}^1_0(\tria^+)$ with finer mesh $\tria^+ \geq \tria$.
\end{remark}
\begin{remark}[Cheaper approaches]
The adaptive loop suggested in this section is much more costly than the adaptive scheme with linear minimal residual methods for $p=2$. This might be a price we have to pay the solve challenging PDE's. On the other hand, for less challenging problems, we might use cheaper versions of the suggested scheme. For example, the scheme performed well in our experiments with fixed relaxation interval and a fixed small number of \Kacanov{} iterations after each mesh refinement as for example done in Section~\ref{sec:ExpErikJohansons}. Alternatively, one might use our scheme only in the last step of an adaptive finite element loop to smoothen oscillations.
\end{remark}
\section{Applications}\label{sec:Applications}
We conclude this paper with an application of our algorithm to convection-diffusion problems. Given a bounded Lipschitz domain $\Omega \subset \mathbb{R}^d$, a diffusion coefficient $\varepsilon > 0$, an incompressible advection field $\beta \in L^\infty(\Omega;\mathbb{R}^d)$, a function $c \in L^\infty(\Omega)$, and a right-hand side $f \in L^2(\Omega)$, this problems seeks $u\in W^{1,2}_0(\Omega)$ with 
\begin{align}\label{eq:convDiff}
- \textup{div}(\varepsilon \nabla u - \beta u ) + cu = f.
\end{align}
Set for all $u\in W^{1,1}_0(\Omega)$ and $v\in W^{1,1}_0(\Omega)$ the functional $F(v) \coloneqq \int_\Omega fv\dx$ and the bilinear form 
\begin{align*}
b(u,v) \coloneqq \int_\Omega \varepsilon \nabla u \cdot \nabla v \dx - \int_\Omega u\beta \cdot \nabla v\dx + \int_\Omega c uv\dx.
\end{align*} 
The variational formulation of~\eqref{eq:convDiff} seeks the solution $u\in W^{1,2}_0(\Omega)$ to
\begin{align}\label{eq:connvectionDiffusion}
b(u,v)  = F(v) \qquad\text{for all }v\in W^{1,2}_0(\Omega).
\end{align}
This formulation allows for the application of our minimal residual method. We therefore discretize the spaces $W_0^{1,p'}(\Omega)$ and $W_0^{1,p}(\Omega)$ with $p \coloneqq 100$ by   
\begin{align*}
U_h \coloneqq \mathcal{L}^1_{1,0}(\tria)\qquad\text{and}\qquad V_h \coloneqq \mathcal{L}^1_{2,0}(\tria).
\end{align*}
Suitable Fortin operators \eqref{eq:Fortin}, which might in fact require higher polynomial degrees in $V_h$, are discussed in \cite[Sec.~4]{MonsuurStevensonStorn23}.
To compare the results, we apply the following alternative schemes:
\begin{enumerate}
\item Our first alternative numerical scheme is the classical Galerkin FEM. It seeks the solution $\fru^G_h\in U_h$ to the problem \label{itm:Galerkin}
\begin{align*}
b(\fru^\textup{G}_h, w_h) = F(w_h) \qquad\text{for all }w_h\in U_h. 
\end{align*}
Adaptive mesh refinements are driven by the standard residual error estimator investigated for example in \cite[Sec.~1.2]{Verfurth1996}.
\item The second alternative is the classical first-order system least squares method \cite{BochevGunzburger09} with Raviart-Thomas space $RT_0(\tria) \coloneqq \lbrace q \in H(\textup{div},\Omega)\colon$ for all $T\in \tria$ exist $A\in \mathbb{R}^d$ and $b\in \mathbb{R}$ with $q(x)|_T = A + b x$ for all $x\in T\rbrace$. It seeks the minimizer $(\fru_h^\textup{LS},\sigma_h^\textup{LS}) \in U_h\times RT_0(\tria)$ that minimizes over all $(u_h,\tau_h) \in U_h\times RT_0(\tria)$ the functional  \label{itm:LS}
\begin{align*}
\lVert \tau_h - \varepsilon \nabla u_h + \beta u_h\rVert_{L^2(\Omega)}^2 + \lVert \textup{div}\, \tau_h - cu_h + f\rVert_{L^2(\Omega)}^2.
\end{align*}
Adaptive mesh refinements are driven by the local contributions
\begin{align*}
\lVert \sigma_h^\textup{LS} - \varepsilon \nabla \fru_h^\textup{LS} + \beta \fru_h^\textup{LS}\rVert_{L^2(T)}^2 + \lVert \textup{div}\, \sigma_h^\textup{LS} - c\fru_h^\textup{LS} + f\rVert_{L^2(T)}^2\quad\text{for all }T\in \tria.
\end{align*}
\item The third alternative is the minimal residual method introduced in \cite[Example~2.2~(i)]{MonsuurStevensonStorn23}, which seeks the solution to the minimization problem \label{itm:MinRes}
\begin{align*}
\fru^\textup{Min}_h = \argmin_{u_h \in U_h} \sup_{v_h\in V_h\setminus \lbrace 0 \rbrace} \frac{b(u_h,v_h) - F(v_h) }{\lVert \nabla v_h \rVert_{L^2(\Omega)}}.
\end{align*}
Adaptive mesh refinements are driven by the local contributions $\lVert \eta_h \rVert_{L^2(T)}^2$ for all $T\in \tria$ of the Riesz representative $\eta_h \in V_h$ with 
\begin{align*}
\int_\Omega \nabla \eta_h\cdot\nabla v_h \dx = b(\fru_h,v_h) - F(v_h) \qquad\text{for all }v_h \in V_h.
\end{align*}
\end{enumerate}
The adaptive schemes use the D\"orfler marking strategy with bulk parameter $0.5$. 
\subsection{Experiment 1 (Viscosity solution 1D)}
\begin{figure}[h]
\begin{tikzpicture}
\begin{axis}[
	ymax = 0.8,    
	cycle multi list={\nextlist MyColorsPlotU},
	legend cell align=left,
	legend style={legend columns=3,legend pos=south east,font=\fontsize{6}{4}\selectfont},
    width = \textwidth,
    height = 0.4\textwidth,
    xlabel = {$x$},]
 
\addplot table [x=x,y=y]{Plots/Exp1_GalerkinNdof31.txt};
\addplot table [x=x,y=y]{Plots/Exp1_LeastSquareNdof31.txt};
\addplot table [x=x,y=y]{Plots/Exp1_MinResL2Ndof31.txt};
\addplot table [x=x,y=y]{Plots/Exp1_p100Ndof31.txt};
\addplot table [x=x,y=y]{Plots/Exp1_InterpolNdof31.txt};
\legend{{$\fru_h^G$},{$\fru_h^\textup{LS}$},{$\fru_h^\textup{Min}$},{$\fru_h$},{$\fru$}};
\end{axis}
\end{tikzpicture}
\caption{Approximations of the viscosity solution $\fru$ to \eqref{eq:viscos1d}.}\label{fig:viscosity1D}
\end{figure}
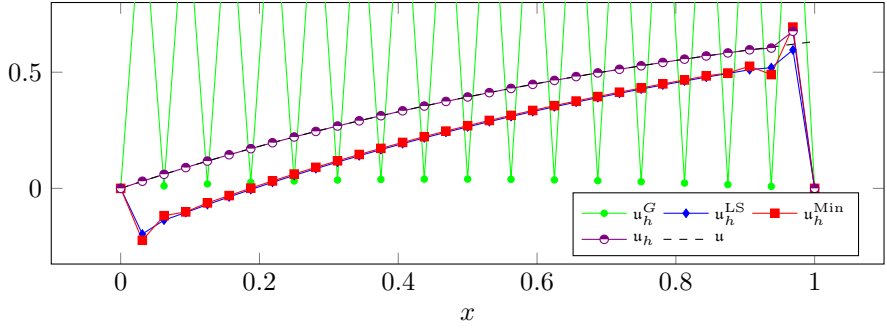%
Our first experiment considers the one dimensional problem 
\begin{align}\label{eq:viscos1d}
\fru' + \fru = 1\text{ in }\Omega \coloneqq (0,1)\qquad\text{with}\qquad \fru(0) = \fru(1) = 0.
\end{align}
This overdetermined ODE has no classical solution, but can be seen as the limiting case $\varepsilon \to 0$ of the problem 
\begin{align*}
-\varepsilon \fru_\varepsilon'' + \fru_\varepsilon' + \fru_\varepsilon = 1\text{ in }\Omega \qquad\text{with}\qquad \fru_\varepsilon(0) = \fru_\varepsilon(1) = 0.
\end{align*}
These functions $\fru_\varepsilon$ converge towards the viscosity solution $\fru(x) = 1 - \exp(-x)$, cf.~\cite[Chap.~1]{Katzourakis15} and \cite[Sec.~4.6]{Guermond04}. Figure~\ref{fig:viscosity1D} displays the resulting approximations on a partition of the unit interval into $2^5$ equidistant intervals. The solution $\fru_h$ to \eqref{eq:discreteMinProb} yields, apart from a tiny oscillation on the last intervals, a very accurate approximation of the viscosity solution. In contrast, the Galerkin FEM results in a highly oscillating function that does not resemble any of the solutions characteristics at all. The solution to the minimization methods in \ref{itm:LS} and \ref{itm:MinRes} experiences some fast decay on the first interval. Thereafter, the approximation increase and experience a similar (but stronger) oscillation at the last two intervals. The accuracy of the solutions to  \ref{itm:Galerkin}--\ref{itm:MinRes} does not improve under uniform mesh refinement. Adaptive mesh refinements, driven by the local residuals of these methods, overcome this problem partially for the methods in \ref{itm:LS}--\ref{itm:MinRes}. Apart from an oscillation near $x=1$, the adaptive LSFEM shows some small oscillation near the origin $x=0$ and the adaptive minimal residual method \ref{itm:MinRes} shows an oscillation near $x=1/2$. This indicates severe difficulties of the methods in \ref{itm:Galerkin}--\ref{itm:MinRes} for problems with small viscosity parameter $\varepsilon \ll 1$, that can be overcome by the minimal residual method~\eqref{eq:discreteMinProb} in this paper.
\subsection{Experiment 2 (Viscosity solution 2D)}
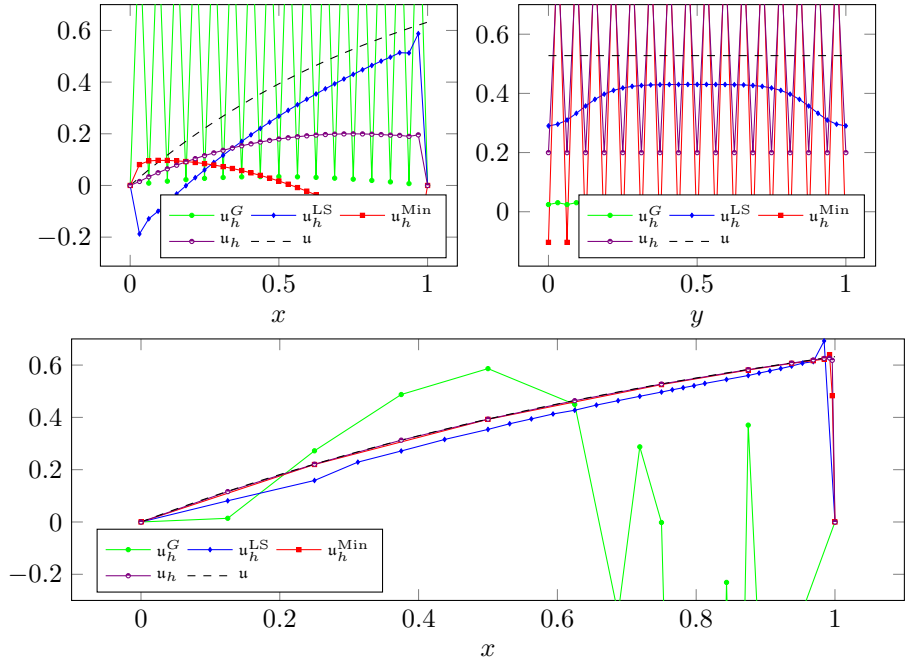
\begin{figure}[h!]
\begin{tikzpicture}
\begin{axis}[
	ymax = 0.7,    
	cycle multi list={\nextlist MyColorsPlotU2},
	legend cell align=left,
	legend style={legend columns=3,legend pos=south east,font=\fontsize{6}{4}\selectfont},
    width = .5\textwidth,
    height = 0.4\textwidth,
    xlabel = {$x$\vphantom{$y$}},]
 
\addplot table [x=x,y=y]{Plots/Exp2_GalerkinNdof1000_unif.txt};
\addplot table [x=x,y=y]{Plots/Exp2_LeastSquareNdof1000_unif.txt};
\addplot table [x=x,y=y]{Plots/Exp2_MinResL2Ndof1000_unif.txt};
\addplot table [x=x,y=y]{Plots/Exp2_p100Ndof1000_unif.txt};
\addplot table [x=x,y=y]{Plots/Exp2_InterpolNdof1000_unif.txt};
\legend{{$\fru_h^G$},{$\fru_h^\textup{LS}$},{$\fru_h^\textup{Min}$},{$\fru_h$},{$\fru$}};
\end{axis}
\end{tikzpicture}
\begin{tikzpicture}
\begin{axis}[
	ymax = 0.7,    
	cycle multi list={\nextlist MyColorsPlotU2},
	legend cell align=left,
	legend style={legend columns=3,legend pos=south east,font=\fontsize{6}{4}\selectfont},
    width = .5\textwidth,
    height = 0.4\textwidth,
    xlabel = {$y$},]
 
\addplot table [x=x,y=y]{Plots/Exp2_GalerkinNdof1000_unif_y.txt};
\addplot table [x=x,y=y]{Plots/Exp2_LeastSquareNdof1000_unif_y.txt};
\addplot table [x=x,y=y]{Plots/Exp2_MinResL2Ndof1000_unif_y.txt};
\addplot table [x=x,y=y]{Plots/Exp2_p100Ndof1000_unif_y.txt};
\addplot table [x=x,y=y]{Plots/Exp2_InterpolNdof1000_unif_y.txt};
\legend{{$\fru_h^G$},{$\fru_h^\textup{LS}$},{$\fru_h^\textup{Min}$},{$\fru_h$},{$\fru$}};
\end{axis}
\end{tikzpicture}
\begin{tikzpicture}
\begin{axis}[
	ymax = 0.7,    
	ymin = -.3,
	cycle multi list={\nextlist MyColorsPlotU2},
	legend cell align=left,
	legend style={legend columns=3,legend pos=south west,font=\fontsize{6}{4}\selectfont},
    width = \textwidth,
    height = 0.4\textwidth,
    xlabel = {$x$},]
 
\addplot table [x=x,y=y]{Plots/Exp2_GalerkinNdof1000_adapt.txt};
\addplot table [x=x,y=y]{Plots/Exp2_LeastSquareNdof1000_adapt.txt};
\addplot table [x=x,y=y]{Plots/Exp2_MinResL2Ndof1000_adapt.txt};
\addplot table [x=x,y=y]{Plots/Exp2_p100Ndof1000_adapt.txt};
\addplot table [x=x,y=y]{Plots/Exp2_InterpolNdof1000_adapt.txt};
\legend{{$\fru_h^G$},{$\fru_h^\textup{LS}$},{$\fru_h^\textup{Min}$},{$\fru_h$},{$\fru$}};
\end{axis}
\end{tikzpicture}
\caption{Approximations with about 1000 degrees of freedom of the viscosity solution $\fru$ with \eqref{eq:viscos2d} evaluated at $(x,1/2)$ (top left) and $(3/4,y)$ (top right) with uniform mesh refinements and at $(x,1/2)$ (bottom) with adaptive mesh refinement.}\label{fig:viscosity2D}
\end{figure}%
In our second experiment we extend the first experiment to two dimensions: We seek the viscosity solution to
\begin{align}\label{eq:viscos2d}
\frac{d}{dx} \fru + \fru = 1\text{ in }\Omega \coloneqq (0,1)^2\qquad\text{with }\fru(0,\bigcdot) = \fru(1,\bigcdot) = 0.
\end{align}
The viscosity solution reads $\fru(x,y) = 1- \exp(-x)$. 
The resulting approximations are displayed in Figure~\ref{fig:viscosity2D}. All methods fail on uniform meshes: While the Galerkin FEM leads to strong oscillations along the $x$-axis, the minimal residual method in \ref{itm:MinRes} and our suggested methods in \eqref{eq:discreteMinProb} lead to strong oscillations along the $y$-axis.
The LSFEM seems to be more robust, but does not provide a good approximation as well. Uniform mesh refinements do not seem to overcome these difficulties.
However, adaptive mesh refinement overcomes this problem for the minimal residual methods. The adaptive LSFEM solution seems to converge to the exact solution but is still much worse than the adaptively computed solutions to the method in \ref{itm:MinRes} and our approach in \eqref{eq:discreteMinProb}. Indeed, the solution to \ref{itm:MinRes} shows only some tiny oscillation near $x=1$, the solution to our scheme in \eqref{eq:discreteMinProb} does not show any oscillation at all and provides a very accurate approximation, cf.~Figure~\ref{fig:viscosity2D}. This shows that adaptivity might be a key in the convergence of our approximation. 
\subsection{Experiment 3 (Eriksson and Johnson)}\label{sec:ExpErikJohansons}
Our third experiment has been introduced by Eriksson and Johnson in \cite{ErikssonJohnson93}. We seek the solution to \eqref{eq:connvectionDiffusion} with $\beta = (1,0)^\top$, right-hand side $f=0$, and initial data $\fru(0,y) = \sin(\pi y)$ and $\fru(x,y)=0$ for $x=1$ or $y\in \{0,1\}$ with unit square domain $\Omega = (0,1)^2$. In other words, we seek the solution to
\begin{align}\label{eq:ErikssonJohnson}
\begin{aligned}
-\varepsilon \Delta \fru + \frac{d}{d x}\fru & = 0&&\text{in }\Omega,\\
\fru(x,y) & = 0&&\text{if }  y\in {0,1}\text{ or }x = 1,\\
\fru(x,y) &  = \sin(\pi y)&&\text{if } x = 0.
\end{aligned}
\end{align}
Let $s_1 \coloneqq (1 + \sqrt{1+4\pi^2\varepsilon^2})(2\varepsilon)^{-1}$ and $s_2 \coloneqq (1 - \sqrt{1+4\pi^2\varepsilon^2})(2\varepsilon)^{-1}$. The exact solution reads 
\begin{align*}
\fru(x,y) = \frac{\exp(s_1(x-1)) - \exp(s_2(x-1))}{\exp(-s_1)-\exp(-s_2)}\sin(\pi y).
\end{align*}
In our first computation we set $\varepsilon \coloneqq 10^{-3}$ and use uniformly refined meshes. In contrast to the previous calculations all methods converge as the mesh is uniformly refined. On coarse grids our minimization schemes leads to superior results as depicted in Figure~\ref{fig:ErikssonJohanson1}. 
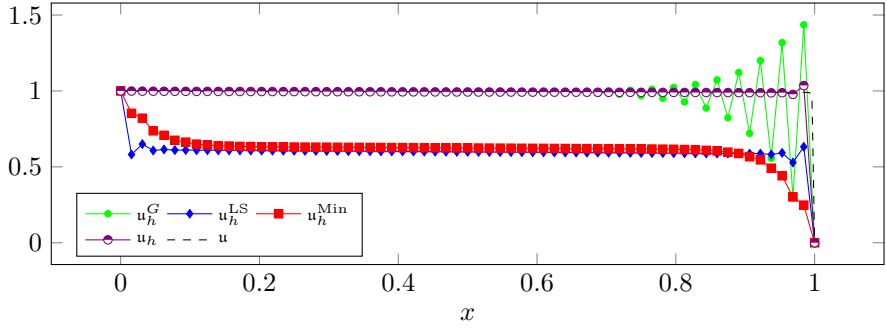
\begin{figure}
\begin{tikzpicture}
\begin{axis}[
	cycle multi list={\nextlist MyColorsPlotU},
	legend cell align=left,
	legend style={legend columns=3,legend pos=south west,font=\fontsize{6}{4}\selectfont},
    width = \textwidth,
    height = 0.4\textwidth,
    xlabel = {$x$},]
 
\addplot table [x=x,y=y]{Plots/Exp3/Exp3_GalerkinNdof9000_unif.txt};
\addplot table [x=x,y=y]{Plots/Exp3/Exp3_LSFEMNdof9000_unif.txt};
\addplot table [x=x,y=y]{Plots/Exp3/Exp3_HMinNdof9000_unif.txt};
\addplot table [x=x,y=y]{Plots/Exp3/Exp3_KacanovNdof9000_unif_w01.txt};
\addplot table [x=x,y=y]{Plots/Exp3/Exp3_InterpolNdof9000_unif.txt};
\legend{{$\fru_h^G$},{$\fru_h^\textup{LS}$},{$\fru_h^\textup{Min}$},{$\fru_h$},{$\fru$}};
\end{axis}
\end{tikzpicture}
\caption{Approximations of the solution to \eqref{eq:ErikssonJohnson} with $\varepsilon = 10^{-3}$, uniform mesh, and $\dim U_h = 4225$ evaluated at $(x,1/2)$.}\label{fig:ErikssonJohanson1}
\end{figure}%

The situation changes drastically when we solve the problem with very small diffusion coefficient $\varepsilon \coloneqq 10^{-6}$. 
For uniform mesh refinements the solutions to the minimal residual method in \eqref{eq:discreteMinProb} and the minimal residual method in \ref{itm:MinRes} show as in Experiment 2 strong oscillations along the $y$-axis.
The direct solver in FEniCS (MUMPS) was not able to solve the resulting system for the Galerkin FEM solution with more than 80 degrees of freedom. The LSFEM solution seems to converge towards a function $u \approx \gamma \, \sin(\pi y)$ with some constant $\gamma$ slight larger than $1/2$ as the mesh is uniformly refined.
Unfortunately, adaptivity does not overcome this problem: 
The direct solver in FEniCS (MUMPS) was not able to compute a solution to the Galerkin FEM and the LSFEM with adaptive mesh refinements for meshes with more than about 300 and 1000 degrees of freedom, respectively.  The adaptive scheme for our method in \eqref{eq:discreteMinProb} refines strongly near $x=0$ and the approximation seems to converge point-wise to zero, cf.~Figure~\ref{fig:ErikssonJohanson2}. The adaptive minimal residual method in \ref{itm:MinRes} refines strongly near $x=0$ and $x=1$ and the approximations look roughly like $u \approx \gamma \, \sin(\pi y)$ with some constant $\gamma$ slightly larger than $1/2$. All in all, non of the schemes converges towards the exact solution. 
We overcome this challenge by slowly adapting the diffusion parameter $\varepsilon$ in our computations in the sense that we set 
\begin{align}\label{eq:adaptedDiffusion}
\varepsilon \coloneqq \begin{cases}
10^{-2}&\text{it }\dim U_h \in [0,1000),\\
10^{-3}&\text{if }\dim U_h \in [1000,5000),\\
10^{-4}&\text{if } \dim U_h \in [5000,10 000),\\
10^{-5} & \text{if }\dim U_h \in [10 000,50 000),\\
10^{-6} & \text{else}.
\end{cases}
\end{align}
Figure~\ref{fig:ErikssonJohanson2} shows the resulting convergence history plot of the error measured in the $L^2(\Omega)$ norm.
Initially adapting $\varepsilon$ helps all methods, but as $\dim U_h$ exceeds $10^3$ (and so $\varepsilon$ is set to $10^{-3}$), the LSFEM starts to struggle.
The same happens for the Galerkin and minimal residual method in \ref{itm:Galerkin} and \ref{itm:MinRes} as $\dim U_h$ exceed $10^4$ (and so $\varepsilon$ is set to $10^{-5}$). In contrast, our method in \eqref{eq:discreteMinProb} still converges as the number of degrees of freedom is increased. In order to save computational power, we did not use the adaptive scheme suggested in Section~\ref{sec:AdaptiveScheme}. Instead, we fixed the relaxation interval $\zeta \coloneqq [10^{-2},10^2]$ and computed only two \Kacanov{} iterations on each mesh. An alternative calculation, using the adaptive strategy in Section~\ref{sec:AdaptiveScheme} with the large weight $w = 100$ causing about five \Kacanov{} iterations on each mesh, led to similar convergence results.
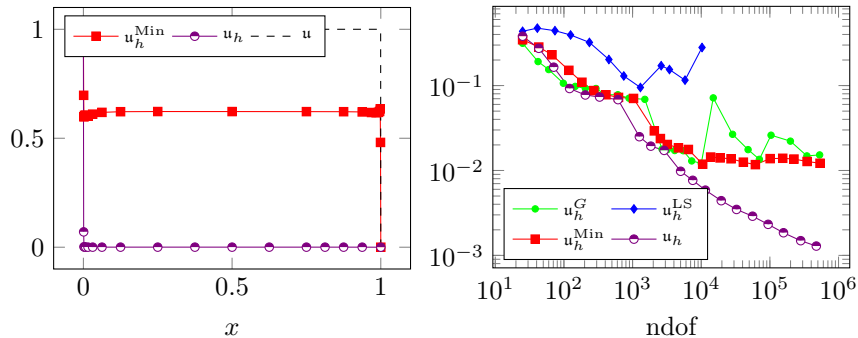
\begin{figure}
\begin{tikzpicture}
\begin{axis}[
	cycle multi list={\nextlist MyColorsPlotU3},
	legend cell align=left,
	legend style={legend columns=3,legend pos=north west,font=\fontsize{6}{4}\selectfont},
    width = .5\textwidth,
    height = 0.4\textwidth,
    xlabel = {$x$\vphantom{$10^3$}},]
 
\addplot table [x=x,y=y]{Plots/Exp3b/Exp3b_Hmin5000.txt};
\addplot table [x=x,y=y]{Plots/Exp3b/Exp3b_Kacndof5000.txt};
\addplot table [x=x,y=y]{Plots/Exp3b/Exp3b_Interpol.txt};
\legend{{$\fru_h^\textup{Min}$},{$\fru_h$},{$\fru$}};
\end{axis}
\end{tikzpicture}
\begin{tikzpicture}
\begin{axis}[
	cycle multi list={\nextlist MyColorsPlotU},
	legend cell align=left,
	legend style={legend columns=2,legend pos=south west,font=\fontsize{6}{4}\selectfont},
	xmode = log,
	ymode = log,
    width = .5\textwidth,
    height = 0.4\textwidth,
    xlabel = {$\textup{ndof}$},]
 
\addplot table [x=ndof,y=L2error]{Plots/Exp3b/Experiment3_Galerkinndof500000_theta_0.5_eps_1e-06.txt};
\addplot table [x=ndof,y=L2error]{Plots/Exp3b/Experiment3_LSFEMndof500000_theta_0.5_eps_1e-06.txt};
\addplot table [x=ndof,y=L2error]{Plots/Exp3b/Experiment3_Hminndof500000_theta_0.5_eps_1e-06.txt};
\addplot table [x=ndof,y=L2error]{Plots/Exp3b/Experiment3_HminKacndof500000_theta_0.5_w_10000_eps_1e-06.txt};
\legend{{$\fru_h^G$},{$\fru_h^\textup{LS}$},{$\fru_h^\textup{Min}$},{$\fru_h$},{$\fru$}};
\end{axis}
\end{tikzpicture}
\caption{The left-hand side shows approximations evaluated at $(x,1/2)$ with adaptive mesh refinements and fixed parameter $\varepsilon = 10^{-6}$ with $\dim U_h \approx 5000$ and the right-hand side shows the convergence history plot of the $L^2(\Omega)$ error for $\varepsilon = 10^{-6}$ with adapted diffusion parameter in \eqref{eq:adaptedDiffusion}.}\label{fig:ErikssonJohanson2}
\end{figure}
\section{Conclusion}
We have introduced a novel numerical scheme that solves minimal residual methods in $W^{-1,p'}(\Omega)$. Additionally, we suggested an iterative scheme that converges towards the discrete solution of the resulting non-linear minimization problem. The scheme converges even for large exponents like $p=100$. The resulting approximations are beneficial for solving challenging PDE's like convection-dominated diffusion problems compared to other schemes like the Galerkin FEM or minimal residual methods in Hilbert spaces. However, in these challenging situations the convergence of our scheme seems to require some suitable mesh design. This can be done adaptively with some suitable designed initial mesh. We thus suggest a scheme where we increase the diffusion parameter depending on the degrees of freedom. This allowed for the approximation of convection-dominated diffusion problems with tiny diffusion parameters like $\varepsilon = 10^{-6}$.

\printbibliography

@article{BalciDieningStorn22,
  doi = {10.48550/ARXIV.2210.06402},
  url = {https://arxiv.org/abs/2210.06402},
  author = {Kh.~Balci, Anna and Diening, Lars and Storn, Johannes},
  SHORTAUTHOR = {Balci, Anna Kh. and Diening, Lars and Storn, Johannes},
  SORTNAME = {Balci, Anna Kh. and Diening, Lars and Storn, Johannes},
  journal = {arXiv},
  title = {Relaxed {K}a{\v c}anov scheme for the $p$-{L}aplacian with large exponent},
  publisher = {arXiv},
  year = {2022},
  copyright = {arXiv.org perpetual, non-exclusive license}
}

@book {BochevGunzburger09,
    AUTHOR = {Bochev, Pavel B. and Gunzburger, Max D.},
     TITLE = {Least-squares finite element methods},
    SERIES = {Applied Mathematical Sciences},
    VOLUME = {166},
 PUBLISHER = {Springer, New York},
      YEAR = {2009},
     PAGES = {xxii+660},
      ISBN = {978-0-387-30888-3},
   MRCLASS = {65-02 (35A35 65M60 65N30 74S05 76M10)},
  MRNUMBER = {2490235},
MRREVIEWER = {Tsu-Fen\ Chen},
       DOI = {10.1007/b13382},
       URL = {https://doi.org/10.1007/b13382},
}

@article {CarstensenDemkowiczGopalakrishnan16,
    AUTHOR = {Carstensen, C. and Demkowicz, L. and Gopalakrishnan, J.},
     TITLE = {Breaking spaces and forms for the {DPG} method and
              applications including {M}axwell equations},
   JOURNAL = {Comput. Math. Appl.},
  FJOURNAL = {Computers \& Mathematics with Applications. An International
              Journal},
    VOLUME = {72},
      YEAR = {2016},
    NUMBER = {3},
     PAGES = {494--522},
      ISSN = {0898-1221},
   MRCLASS = {65N30 (65N12 78A25)},
  MRNUMBER = {3521055},
MRREVIEWER = {Stefan Kurz},
       DOI = {10.1016/j.camwa.2016.05.004},
       URL = {https://doi.org/10.1016/j.camwa.2016.05.004},
}

@article{CarstensenDemkowiczGopalakrishnan14,
	TITLE = {A posteriori error control for {DPG} methods},
	AUTHOR = {Carstensen, C. and Demkowicz, L. and Gopalakrishnan, J.} ,
        JOURNAL = {SIAM J. Numer. Anal.},
	FJOURNAL = {SIAM Journal on Numerical Analysis},
	VOLUME = {52},
	NUMBER = {3},
	YEAR = {2014},
	PAGES = {1335-1353},
	DOI = {10.1137/130924913},
        Xbtbrowserfulltext = {../download/2014-CC_LD_JG-aposteriori_error_control_dpg_methods.pdf},
}

@article {DieningFornasierTomasiWank20,
    AUTHOR = {Diening, L. and Fornasier, M. and Tomasi, R. and Wank, M.},
     TITLE = {A relaxed {K}a\v{c}anov iteration for the {$p$}-{P}oisson problem},
   JOURNAL = {Numer. Math.},
  FJOURNAL = {Numerische Mathematik},
    VOLUME = {145},
      YEAR = {2020},
    NUMBER = {1},
     PAGES = {1--34},
      ISSN = {0029-599X},
   MRCLASS = {65N30 (35J70 35J92 65N12 65N22)},
  MRNUMBER = {4091593},
       DOI = {10.1007/s00211-020-01107-1},
       URL = {https://doi.org/10.1007/s00211-020-01107-1},
}

@article {DemlowLeykekhmanSchatzWahlbin12,
    AUTHOR = {Demlow, A. and Leykekhman, D. and Schatz, A. H. and Wahlbin,
              L. B.},
     TITLE = {Best approximation property in the {$W^{1}_{\infty}$} norm for
              finite element methods on graded meshes},
   JOURNAL = {Math. Comp.},
  FJOURNAL = {Mathematics of Computation},
    VOLUME = {81},
      YEAR = {2012},
    NUMBER = {278},
     PAGES = {743--764},
      ISSN = {0025-5718},
   MRCLASS = {65N30 (65N15 65N50)},
  MRNUMBER = {2869035},
MRREVIEWER = {Nicolae Pop},
       DOI = {10.1090/S0025-5718-2011-02546-9},
       URL = {https://doi.org/10.1090/S0025-5718-2011-02546-9},
}

@article{DieningRolfesSalgado23,
      title={Pointwise gradient estimate of the {R}itz projection}, 
      author={Lars Diening and Julian Rolfes and Abner J. Salgado},
      year={2023},
      journal={arXiv preprint 2305.03575},
      archivePrefix={arXiv},
      primaryClass={math.NA},
      publisher = {arXiv},
      url = {https://doi.org/10.48550/arXiv.2305.03575},
      doi = {10.48550/arXiv.2305.03575},
}

@article {ErikssonJohnson93,
    AUTHOR = {Eriksson, Kenneth and Johnson, Claes},
     TITLE = {Adaptive streamline diffusion finite element methods for
              stationary convection-diffusion problems},
   JOURNAL = {Math. Comp.},
  FJOURNAL = {Mathematics of Computation},
    VOLUME = {60},
      YEAR = {1993},
    NUMBER = {201},
     PAGES = {167--188, S1--S2},
      ISSN = {0025-5718},
   MRCLASS = {65N15 (65N30 76M10 76Rxx)},
  MRNUMBER = {1149289},
       DOI = {10.2307/2153160},
       URL = {https://doi.org/10.2307/2153160},
}

@article {Guermond04,
    AUTHOR = {Guermond, J. L.},
     TITLE = {A finite element technique for solving first-order {PDE}s in
              {$L^P$}},
   JOURNAL = {SIAM J. Numer. Anal.},
  FJOURNAL = {SIAM Journal on Numerical Analysis},
    VOLUME = {42},
      YEAR = {2004},
    NUMBER = {2},
     PAGES = {714--737},
      ISSN = {0036-1429},
   MRCLASS = {65N30},
  MRNUMBER = {2084233},
MRREVIEWER = {Bruce A. Finlayson},
       DOI = {10.1137/S0036142902417054},
       URL = {https://doi.org/10.1137/S0036142902417054},
}

@article {Hanner56,
    AUTHOR = {Hanner, Olof},
     TITLE = {On the uniform convexity of {$L^p$} and {$l^p$}},
   JOURNAL = {Ark. Mat.},
  FJOURNAL = {Arkiv f\"{o}r Matematik},
    VOLUME = {3},
      YEAR = {1956},
     PAGES = {239--244},
      ISSN = {0004-2080,1871-2487},
   MRCLASS = {46.2X},
  MRNUMBER = {77087},
MRREVIEWER = {M.\ M.\ Day},
       DOI = {10.1007/BF02589410},
       URL = {https://doi.org/10.1007/BF02589410},
}

@article{HoustonMugaRoggendorfvanderZee19,
url = {https://doi.org/10.1515/cmam-2018-0198},
title = {The Convection-Diffusion-Reaction Equation in Non-{H}ilbert {S}obolev Spaces: A Direct Proof of the Inf-Sup condition and stability of {G}alerkin’s Method},
author = {Paul Houston and Ignacio Muga and Sarah Roggendorf and Kristoffer G. van der Zee},
pages = {503--522},
volume = {19},
number = {3},
journal = {Computational Methods in Applied Mathematics},
doi = {doi:10.1515/cmam-2018-0198},
year = {2019},
lastchecked = {2023-06-28}
}

@article{HoustonRoggendorfVanDerZee22,
    AUTHOR = {Houston, Paul and Roggendorf, Sarah and van der Zee,
              Kristoffer G.},
     TITLE = {Gibbs phenomena for {${\rm L}^q$}-best approximation in finite
              element spaces},
   JOURNAL = {ESAIM Math. Model. Numer. Anal.},
  FJOURNAL = {ESAIM. Mathematical Modelling and Numerical Analysis},
    VOLUME = {56},
      YEAR = {2022},
    NUMBER = {1},
     PAGES = {177--211},
      ISSN = {2822-7840,2804-7214},
   MRCLASS = {65N30 (41A50)},
  MRNUMBER = {4376273},
MRREVIEWER = {Gerrit Welper},
       DOI = {10.1051/m2an/2021086},
       URL = {https://doi.org/10.1051/m2an/2021086},
}

@article {JerisonKenig95,
    AUTHOR = {Jerison, David and Kenig, Carlos E.},
     TITLE = {The inhomogeneous {D}irichlet problem in {L}ipschitz domains},
   JOURNAL = {J. Funct. Anal.},
  FJOURNAL = {Journal of Functional Analysis},
    VOLUME = {130},
      YEAR = {1995},
    NUMBER = {1},
     PAGES = {161--219},
      ISSN = {0022-1236,1096-0783},
   MRCLASS = {35J25 (46E35)},
  MRNUMBER = {1331981},
MRREVIEWER = {H.\ Triebel},
       DOI = {10.1006/jfan.1995.1067},
       URL = {https://doi.org/10.1006/jfan.1995.1067},
}

@book {Katzourakis15,
    AUTHOR = {Katzourakis, Nikos},
     TITLE = {An introduction to viscosity solutions for fully nonlinear
              {PDE} with applications to calculus of variations in {$L\sp
              \infty$}},
    SERIES = {SpringerBriefs in Mathematics},
 PUBLISHER = {Springer, Cham},
      YEAR = {2015},
     PAGES = {xii+123},
      ISBN = {978-3-319-12828-3},
   MRCLASS = {35-01 (35D40 35F20 35J60 49L25)},
  MRNUMBER = {3289084},
MRREVIEWER = {Enea\ Parini},
       DOI = {10.1007/978-3-319-12829-0},
       URL = {https://doi.org/10.1007/978-3-319-12829-0},
}

@article{LiDemkowicz22,
url = {https://doi.org/10.1515/cmam-2021-0158},
title = {An $L^p$-DPG Method with Application to 2D Convection-Diffusion Problems},
author = {Jiaqi Li and Leszek Demkowicz},
pages = {649--662},
volume = {22},
number = {3},
journal = {Computational Methods in Applied Mathematics},
doi = {doi:10.1515/cmam-2021-0158},
year = {2022},
lastchecked = {2023-06-28}
}

@article{MillarMugaRojasVanDerZee22,
  title={Projection in negative norms and the regularization of rough linear functionals},
  author={Millar, Felipe and Muga, Ignacio and Rojas, Sergio and Van der Zee, Kristoffer G},
  journal={Numerische Mathematik},
  volume={150},
  number={4},
  pages={1087--1121},
  year={2022},
  publisher={Springer},
  DOI = {10.1007/s00211-022-01278-z},
}

@article{MonsuurStevensonStorn23,
  doi = {10.48550/ARXIV.2301.10484},
  url = {https://arxiv.org/abs/2301.10484},
  author = {Monsuur, Harald and Stevenson, Rob and Storn, Johannes},
  title = {Minimal residual methods in negative or fractional Sobolev norms},
  publisher = {arXiv},
  year = {2023},
  journal = {arXiv},
  copyright = {arXiv.org perpetual, non-exclusive license}
}

@article {MugaZee20,
    AUTHOR = {Muga, Ignacio and van der Zee, Kristoffer G.},
     TITLE = {Discretization of linear problems in {B}anach spaces: residual
              minimization, nonlinear {P}etrov-{G}alerkin, and monotone
              mixed methods},
   JOURNAL = {SIAM J. Numer. Anal.},
  FJOURNAL = {SIAM Journal on Numerical Analysis},
    VOLUME = {58},
      YEAR = {2020},
    NUMBER = {6},
     PAGES = {3406--3426},
      ISSN = {0036-1429},
   MRCLASS = {65N30 (41A65 46B20 65J05 65N12 65N15)},
  MRNUMBER = {4178383},
MRREVIEWER = {Andreas Petersson},
       DOI = {10.1137/20M1324338},
       URL = {https://doi.org/10.1137/20M1324338},
}

@article{Storn20,
    AUTHOR = {Storn, Johannes},
     TITLE = {On a relation of discontinuous {P}etrov-{G}alerkin and
              least-squares finite element methods},
   JOURNAL = {Comput. Math. Appl.},
  FJOURNAL = {Computers \& Mathematics with Applications. An International
              Journal},
    VOLUME = {79},
      YEAR = {2020},
    NUMBER = {12},
     PAGES = {3588--3611},
      ISSN = {0898-1221,1873-7668},
   MRCLASS = {65N30 (65N12)},
  MRNUMBER = {4094782},
       DOI = {10.1016/j.camwa.2020.02.018},
       URL = {https://doi.org/10.1016/j.camwa.2020.02.018},
}

@book{Verfurth1996,
  title={A Review of Posteriori Error Estimation and Adaptive Mesh-Refinement Techniques},
  author={Verf{\"u}rth, R.},
  isbn={9780471967958},
  lccn={gb96047389},
  url={https://books.google.de/books?id=d-mwzgEACAAJ},
  year={1996},
  publisher={Wiley}
}

@book {Zeidler95,
    AUTHOR = {Zeidler, Eberhard},
     TITLE = {Applied functional analysis},
    SERIES = {Applied Mathematical Sciences},
    VOLUME = {109},
      NOTE = {Main principles and their applications},
 PUBLISHER = {Springer-Verlag, New York},
      YEAR = {1995},
     PAGES = {xvi+404},
      ISBN = {0-387-94422-2},
   MRCLASS = {00A05 (35-01 46-01 47-01 49-01 58-01)},
  MRNUMBER = {1347692},
MRREVIEWER = {Jean\ Mawhin},
}

\end{document}